\documentclass[a4paper]{article}
\usepackage[utf8]{inputenc}

 \pdfoutput=1

\usepackage{amsmath, amsthm, amsfonts, amssymb}

\usepackage[multiple]{footmisc}

\usepackage{graphicx}
\usepackage[english]{babel}
\usepackage{float}
\usepackage{authblk}

\usepackage[usenames,dvipsnames]{color}
\usepackage{fullpage}
\usepackage{url}

\newcommand\R{\ensuremath{\mathbb{R}}}
\newcommand\N{\ensuremath{\mathbb{N}}}
\newcommand\Z{\ensuremath{\mathbb{Z}}}

\newcommand\vv{\ensuremath{\mathbf{v}}}

\newcommand\M{\mathcal{M}}

\newcommand{\pth}[1]{\left( #1 \right)}

\DeclareMathOperator{\Or}{or}

\newcommand{\simplex}[1]{\Delta_{#1}}
\newcommand{\skel}[2]{#2^{(#1)}}
\newcommand{\skelsim}[2]{\skel{#1}{\simplex{#2}}}

\newtheorem{theorem}{Theorem}
\newtheorem{lemma}[theorem]{Lemma}

\theoremstyle{remark}

\newtheorem{remark}[theorem]{Remark}

\theoremstyle{plain}
\newtheorem{claim}[theorem]{Claim}

\newtheorem{conjecture}[theorem]{Conjecture}

\theoremstyle{remark}

\newtheorem*{remarks*}{Remarks}
\newtheorem*{acknowledgement}{Acknowledgement}

\DeclareMathOperator{\card}{card}

\DeclareMathOperator{\simp}{simp}
\DeclareMathOperator{\sing}{sing}

\newcommand{\gsimp}{\ensuremath{g_{\simp}}}

\newcounter{sideremark}

\title{On 
Generalized Heawood Inequalities for Manifolds: a van Kampen--Flores-type Nonembeddability Result%
\footnote{The work by Z.P.\ was partially supported by the Israel Science Foundation grant ISF-768/12.  
The work by Z.P.\ and M.T.\ was partially supported by the
project CE-ITI (GACR P202/12/G061) of the Czech Science Foundation and by the ERC Advanced Grant No. 267165. 
Part of the research work of M.T. was conducted at IST Austria, supported by an
\emph{IST Fellowship}. The research of P.P. was supported by the ERC
Advanced grant no. 320924.
The work by I.M.\ and U.W.\ was supported by the Swiss National Science
Foundation (grants SNSF-200020-138230 and SNSF-PP00P2-138948).}
{ }\footnote{An extended abstract of this work was presented at the 31st International Symposium on Computational Geometry ~\cite{Kuhnel_conference_version}.
}}

\date{}

\author[1]{Xavier Goaoc}
\author[2]{Isaac Mabillard}
\author[3]{Pavel Pat\'ak}
\author[3]{Zuzana Pat\'akov\'a}
\author[4]{Martin Tancer}
\author[2]{Uli Wagner}
\affil[1]{\small  Universit\'e Paris-Est, LIGM (UMR 8049), CNRS, ENPC, ESIEE, UPEM, F-77454, Marne-la-Vall\'ee, France
}
\affil[2]{\small IST Austria, 
Am Campus 1, 3400 
Klosterneuburg, Austria
}
\affil[3]{\small Einstein Institute of Mathematics, Hebrew University of Jerusalem, Givat Ram, Jerusalem, 9190401, Israel
}
\affil[4]{\small Department of Applied Mathematics 
and Institute for Theoretical Computer Science, 
  Charles University,
Malostransk\'e n\'am. 25, 118 00 Praha 1, 
Czech Republic
}

\begin{document}
\maketitle

\begin{abstract}
  The fact that the complete graph $K_5$ does not embed in the plane
  has been generalized in two independent directions.  On the one
  hand, the solution of the classical \emph{Heawood problem} for
  graphs on surfaces established that the complete graph $K_n$ embeds
  in a closed surface $M$ (other than the Klein bottle) if and only if
  $(n-3)(n-4)\leq 6b_1(M)$, where $b_1(M)$ is the first $\Z_2$-Betti
  number of $M$. On the other hand, van Kampen and Flores proved that
  the $k$-skeleton of the $n$-dimensional simplex (the
  higher-dimensional analogue of $K_{n+1}$) embeds in $\R^{2k}$ if and
  only if~$n \le 2k+1$.

 Two decades ago, K\"uhnel conjectured that the $k$-skeleton of the $n$-simplex embeds in a compact, $(k-1)$-connected $2k$-manifold with $k$th $\Z_2$-Betti number $b_k$ only if the following \emph{generalized Heawood inequality} holds: $\binom{n-k-1}{k+1} \le \binom{2k+1}{k+1}b_k$. This is a common generalization of the case of graphs on surfaces as well as the van Kampen--Flores theorem.

In the spirit of K\"uhnel's conjecture, we prove that if the
$k$-skeleton of the $n$-simplex embeds in a $2k$-manifold with $k$th
$\Z_2$-Betti number $b_k$, then $n \le 2b_k\binom{2k+2}{k} + 2k + 4$.
This bound is weaker than the generalized Heawood inequality, but does
not require the assumption that $M$ is $(k-1)$-connected. Our results
generalize to maps without $q$-covered points, in the spirit of
Tverberg's theorem, for q a prime power. Our proof uses a result of
Volovikov about maps that satisfy a certain homological triviality
condition.

   \end{abstract}

\section{Introduction}

Given a closed surface $M$, a natural question
  is to determine the maximum integer $n$ such that the complete graph
$K_n$ can be embedded (drawn without crossings) into $M$ (e.g., $n=4$
if $M=S^2$ is the $2$-sphere, and $n=7$ if $M$ is a torus). This
classical problem
was raised in the late $19$th century by Heawood \cite{Heawood:MapColourTheorem-1890} and Heffter \cite{Heffter:Nachbargebiete-1891} and completely settled in the 1950--60's through a sequence of works by Gustin, Guy, Mayer, Ringel, Terry, Welch, and Youngs (see \cite[Ch.~1]{Ringel:MapColorTheorem-1974} for a discussion of the history of the problem and detailed references). Heawood already observed that if $K_n$ embeds into $M$ then 
\begin{equation}
\label{eq:thread-problem}
\hfill 
(n-3)(n-4)\leq 6b_1(M)=12-6\chi(M),
\hfill
\end{equation}
where $\chi(M)$ is the Euler characteristic of $M$ and $b_1(M)=2-\chi(M)$ is
the first $\Z_2$-\emph{Betti number} of $M$, i.e., the dimension of the first
homology group $H_1(M;\Z_2)$ (here and throughout the paper, we work with
homology with $\Z_2$-coefficients).
Conversely, for surfaces $M$ other than
the Klein bottle, the inequality is tight, i.e., $K_n$ embeds into $M$ if and
only if \eqref{eq:thread-problem} holds; this is a hard result, the bulk of the
monograph \cite{Ringel:MapColorTheorem-1974} is devoted to its proof. (The
exceptional case, the Klein bottle, has $b_1=2$, but does not admit an
embedding of $K_7$, only of $K_6$.)\footnote{The inequality
\eqref{eq:thread-problem}, which by a direct calculation is equivalent to
$n\leq c(M):=\lfloor (7+\sqrt{1+24\beta_1(M)})/2\rfloor $, is closely related
to the \emph{Map Coloring Problem} for surfaces (which is the context in which
Heawood originally considered the question). Indeed, it turns out that for
surfaces $M$ \emph{other than the Klein bottle}, $c(M)$ is the maximum
chromatic number of any graph embeddable into $M$. For $M=S^2$ the $2$-sphere
(i.e., $b_1(M)=0$), this is the \emph{Four-Color Theorem}
\cite{Appel:Every-planar-map-is-four-colorable.-I.-Discharging-1977,Appel:Every-planar-map-is-four-colorable.-II.-Reducibility-1977};
for other surfaces (i.e., $b_1(M)>0$) this was originally stated (with an
incomplete proof) by Heawood and is now known as the \emph{Map Color Theorem}
or \emph{Ringel--Youngs Theorem} \cite{Ringel:MapColorTheorem-1974}.
Interestingly, for surfaces $M\neq S^2$, there is a fairly short proof, based
on edge counting and Euler characteristic, that the chromatic number of any
graph embeddable into $M$ is at most $c(M)$ (see \cite[Thms.~4.2 and
4.8]{Ringel:MapColorTheorem-1974}) whereas the hard part is the tightness
of~\eqref{eq:thread-problem}.
} 

\smallskip

The question naturally generalizes to higher dimension: Let
$\skelsim{k}{n}$ denote the $k$-skeleton of the $n$-simplex, the
natural higher-dimensional generalization of $K_{n+1}=\skelsim{1}{n}$
(by definition $\skelsim{k}{n}$ has $n+1$ vertices and every subset of
at most $k+1$ vertices forms a face). Given a $2k$-dimensional manifold
$M$, what is the largest $n$ such that $\skelsim{k}{n}$ embeds
(topologically) into $M$? This line of enquiry started in the $1930$'s
when van Kampen~\cite{vanKampen:KomplexeInEuklidischenRaeumen-1932}
and Flores~\cite{Flores:NichtEinbettbar-1933} showed that
$\skelsim{k}{2k+2}$ does not embed into $\R^{2k}$ (the case $k=1$
corresponding to the non-planarity of $K_5$). Somewhat surprisingly,
little else seems to be known, and the following conjecture of
K{\"u}hnel~\cite[Conjecture~B]{Kuhnel:Manifolds-in-the-skeletons-of-convex-polytopes-tightness-and-generalized-Heawood-inequalities-1994}
regarding a \emph{generalized Heawood inequality} remains
  unresolved:

\begin{conjecture}[K\"uhnel]
  \label{c:kuhnel}
  Let $n,k\geq 1$ be integers. If $\skelsim{k}{n}$ embeds in a
  compact, $(k-1)$-connected $2k$-manifold $M$ with $k$th $\Z_2$-Betti
  number $b_k(M)$ then
  \begin{equation} \label{eq:generalized-heawood} \hfill
    \binom{n-k-1}{k+1} \le \binom{2k+1}{k+1}b_k(M). \hfill
  \end{equation}
\end{conjecture}

\noindent
The classical Heawood inequality \eqref{eq:thread-problem} and the van
Kampen--Flores Theorem correspond the special cases $k=1$ and $b_k=0$,
respectively. K\"uhnel states Conjecture~\ref{c:kuhnel} in slightly
different form, in terms of Euler characteristic of $M$ rather than
$b_k(M)$; our reformulation is equivalent. The
$\Z_2$-coefficients are not important in the statement of the
conjecture but they are convenient for our new developments.

\subsection{New results toward K\"{u}hnel's conjecture}

In this paper, our main result is an estimate, in the spirit of the
generalized Heawood inequality \eqref{eq:generalized-heawood}, on the
largest $n$ such that $|\skelsim k{n}|$ almost embeds into a given
$2k$-dimensional manifold. An almost embedding is a relaxation of the
notion of embedding that is useful in setting up our proof method.

Let $K$ be a finite simplicial complex and let $|K|$ be its underlying
space (geometric realization). We define an \emph{almost-embedding} of
$K$ into a (Hausdorff) topological space $X$ to be a continuous map
$f:|K| \to X$ such that any two disjoint simplices $\sigma, \tau \in
K$ have disjoint images, $f(|\sigma|)\cap f(|\tau|)=\emptyset$. In
particular, every embedding is an almost-embedding as well. Let us
stress, however, that the condition for being an almost-embedding
depends on the actual simplicial complex (the triangulation), not just
the underlying space.  That is, if $K$ and $L$ are two different
complexes with $|K| = |L|$ then a map $f:|K|=|L| \to X$ may be an
almost-embedding of $K$ into $X$ but not an almost-embedding of $L$
into~$X$. Our main result is the following.

\begin{theorem}\label{t:kuhnel-2}
  If $\skelsim{k}{n}$ almost embeds into a $2k$-manifold $M$ then
  \[ n \le 2 \binom{2k+2}{k} b_k(M)+ 2k+4,\]
  where $b_k(M)$ is the $k$th $\Z_2$-Betti number of $M$.
\end{theorem}

\noindent
This bound is weaker than the conjectured generalized Heawood
inequality \eqref{eq:generalized-heawood} and is clearly not optimal
(as we already see in the special cases $k=1$ or $b_k(M)=0$).

Apart from applying more generally to almost embeddings, the
hypotheses of Theorem~\ref{t:kuhnel-2} are weaker than those of
Conjecture~\ref{c:kuhnel} in that we do not assume the manifold $M$ to
be $(k-1)$-connected. We conjecture that this connectedness assumption
is not necessary for Conjecture~\ref{c:kuhnel}, i.e., that
\eqref{eq:generalized-heawood} holds whenever $\skelsim{k}{n}$ almost
embeds into a $2k$-manifold $M$. The intuition is that
$\skelsim{k}{n}$ is $(k-1)$-connected and therefore the image of an
almost-embedding cannot ``use'' any parts of $M$ on which nontrivial
homotopy classes of dimension less than $k$ are supported.

\paragraph{Previous work.}

The following special case of Conjecture~\ref{c:kuhnel} was proved by
K\"uhnel~\cite[Thm.~2]{Kuhnel:Manifolds-in-the-skeletons-of-convex-polytopes-tightness-and-generalized-Heawood-inequalities-1994}
(and served as a motivation for the general conjecture): Suppose that
$P$ is an $n$-dimensional simplicial convex polytope, and that there
is a subcomplex of the boundary $\partial P$ of $P$ that is
\emph{$k$-Hamiltonian} (i.e., that contains the $k$-skeleton of $P$)
and that is a triangulation of $M$, a $2k$-dimensional manifold. Then
inequality \eqref{eq:generalized-heawood} holds. To see that this is
indeed a special case of Conjecture~\ref{c:kuhnel}, note that
$\partial P$ is a \emph{piecewise linear} (\emph{PL}) sphere of
dimension $n-1$, i.e., $\partial P$ is combinatorially isomorphic to
some subdivision of $\partial \simplex{n}$ (and, in particular,
$(n-2)$-connected).  Therefore, the $k$-skeleton of $P$, and hence
$M$, contains a subdivision of $\skelsim{k}{n}$ and is
$(k-1)$-connected.

In this special case and for $n\geq 2k+2$, equality in
\eqref{eq:generalized-heawood} is attained if and only if $P$ is a
simplex.  More generally, equality is attained whenever $M$ is a
triangulated $2k$-manifold on $n+1$ vertices that is
\emph{$(k+1)$-neighborly} (i.e., any subset of at most $k+1$ vertices
forms a face, in which case $\skelsim{k}{n}$ is a subcomplex of $M$).
Some examples of $(k+1)$-neighborly $2k$-manifolds are known, e.g.,
for $k=1$ (the so-called \emph{regular cases} of equality for the
Heawood inequality \cite{Ringel:MapColorTheorem-1974}), for $k=2$
\cite{Kuhnel:The-unique-3-neighborly-4-manifold-with-1983,Kuhnel:The-9-vertex-complex-projective-plane-1983}
(e.g., a $3$-neighborly triangulation of the complex projective plane)
and for $k=4$
\cite{Brehm:15-vertex-triangulations-of-an-8-manifold-1992}, but in
general, a characterization of the higher-dimensional cases of
equality for \eqref{eq:generalized-heawood} (or even of those values
of the parameters for which equality is attained) seems rather hard
(which is maybe not surprising, given how difficult the construction
of examples of equality already is for $k=1$).

\subsection{Generalization to points covered $q$ times}

K\"uhnel's conjecture can be recast in a broader setting suggested by
a generalization by Sarkaria~\cite[Thm 1.5]{Sarkaria:vanKampen} of the
van Kampen--Flores Theorem.  Sarkaria's theorem states that if $q$ is
a prime, and $d$ and $k$ integers such that $d \leq \frac{q}{q-1}k$,
then for every continuous map $f\colon |\skelsim k{qk+2q-2}| \to \R^d$
there are $q$ pairwise disjoint simplices $\sigma_1, \dots, \sigma_{q}
\in K$ with intersecting images $f(|\sigma_1|) \cap \dots \cap
f(|\sigma_q|) \neq \emptyset$. Sarkaria's result was generalized by
Volovikov \cite{Volovikov:On-the-van-Kampen-Flores-theorem-1996} for
$q$ being a prime power.

Define an \emph{almost $q$-embedding} of $K$ into a (Hausdorff)
topological space $X$ as a continuous map $f:|K| \to X$ such that any
$q$ pairwise disjoint faces $\sigma_1, \dots, \sigma_{q} \in K$ have
disjoint images $f(|\sigma_1|) \cap \dots \cap f(|\sigma_q|) =
\emptyset$. (So almost $2$-embeddings are almost embeddings.) Again,
being an almost $q$-embedding depends on the actual simplicial complex
(the triangulation), not just the underlying space. Our proof
technique also gives an estimate for almost $q$-embeddings when $q$ is
a prime power.

\begin{theorem}\label{t:kuhnel-general}
  Let $q=p^m$ be a prime power. If $\skelsim{k}{n}$ $q$-almost-embeds
  into a $d$-manifold $M$ with $d \le \frac{q}{q-1}k$ then
  \[ n \le \pth{(q-2)k+2q-2} \binom{qk+2q-2}{k} b_k(M)+(2q-2)k+4q-4,\]
  where $b_k(M)$ is the $k$th $\Z_p$-Betti number of $M$.
\end{theorem}

\noindent
Theorem~\ref{t:kuhnel-general} specializes for $q=2$ to Theorem~\ref{t:kuhnel-2}.

\subsection{Proof technique}

Our proof of Theorem~\ref{t:kuhnel-general} strongly relies on a
different generalization of the van Kampen--Flores Theorem, due to
Volovikov~\cite{Volovikov:On-the-van-Kampen-Flores-theorem-1996}
regarding maps into general manifolds but under an additional
homological triviality condition.

\begin{theorem}[Volovikov]
  \label{t:volovikov_q}
  Let $q = p^m$ be a prime power. Let $f\colon \skelsim k{qk + 2q-2}
  \to M$ be a continuous map where $M$ is a compact $d$-manifold
  with $d \le \frac{q}{q-1}k$. If the induced homomorphism 
  \[ f_*\colon H_{k}\pth{\skelsim{k}{qk+2q-2};\Z_p} \to
  H_{k}(M;\Z_p)\]
  is trivial then $f$ is not a $q$-almost embedding.
\end{theorem}

\noindent
Theorem~\ref{t:volovikov_q} is obtained by specializing the corollary
in Volovikov's
article~\cite{Volovikov:On-the-van-Kampen-Flores-theorem-1996} to $m =
d$ and $s = k+1$.
Note that
Volovikov~\cite{Volovikov:On-the-van-Kampen-Flores-theorem-1996}
formulates the triviality condition in terms of cohomology, i.e., he
requires that $f^*: H^{k}(M;\Z_p) \to H^{k}(\skelsim{k}{2k+2};\Z_p)$
is trivial. However, since we are working with field coefficients and
the (co)homology groups in question are finitely generated, the
homological triviality condition (which is more convenient for us to
work with) and the cohomological one are equivalent.\footnote{More
  specifically, by the Universal Coefficient Theorem
  \cite[53.5]{Munkres:AlgebraicTopology-1984}, $H_k({\,\cdot\,}
  ;\Z_p)$ and $H^k({\,\cdot\,} ;\Z_p)$ are dual vector spaces, and
  $f^\ast$ is the adjoint of $f_\ast$, hence triviality of $f_\ast$
  implies that of $f^\ast$.  Moreover, if the homology group $H_k(X
  ;\Z_p)$ of a space $X$ is finitely generated (as is the case for
  both $\skelsim{k}{n}$ and $M$, by assumption) then it is
  (non-canonically) isomorphic to its dual vector space $H^k(X
  ;\Z_p)$. Therefore, $f_\ast$ is trivial if and only if $f^\ast$ is.}
Note that the homological triviality condition is automatically
satisfied if $H_k(M;\Z_p)=0$, e.g., if $M=\R^{2k}$ or $M=S^{2k}$.  On
the other hand, without the homological triviality condition, the
assertion is in general not true for other manifolds (e.g., $K_5$
embeds into every closed surface different from the sphere, or
$\skelsim{2}{8}$ embeds into the complex projective plane).

\bigskip

The key idea of our approach is to show that if $n$ is large enough and $f$ is
a mapping from $\skelsim{k}{n}$ to $M$, then there is a $q$-almost-embedding $g$ 
from $\skelsim ks$ to $|\skelsim kn|$ for some prescribed value of $s$ such that 
the composed map $f \circ g\colon \Delta_s \to M$ satisfies Volovikov's condition. 
More specifically, the following is our main technical lemma:

\begin{lemma}\label{l:chain_p}
  Let $k,s\geq 1$ and $b\geq 0$ be integers. Let $p$ be a prime
  number. There exists a value $n_0 := n_0(k,b,s,p)$ with the
  following property. Let $n \geq n_0$ and let $f$ be a mapping of
  $|\skelsim{k}{n}|$ into a manifold $M$ with $k$th $\Z_p$-Betti
  number at most $b$.  Then there exists a subdivision $D$ of
  $\skelsim{k}{s}$ and a simplicial map $\gsimp\colon D \to
  \skelsim{k}{n}$ with the following properties.
  \begin{enumerate}
  \item The induced map on the geometric realizations $g\colon |D|
    = |\skelsim{k}{s}| \to |\skelsim{k}{n}|$ is an
    almost-embedding from $\skelsim{k}{s}$ to $|\skelsim{k}{n}|$.
  \item The homomorphism $(f\circ g)_\ast:H_{k}(\skelsim{k}{s}; \Z_p)
    \to H_k(M; \Z_p)$ is trivial (see Section~\ref{s:prelim} below for
    the precise interpretation of $(f\circ g)_\ast$).
  \end{enumerate}
  The value $n_0$ can be taken as $\binom{s}{k}b(s-2k) + 2s-2k+1$.
\end{lemma}

\noindent
Therefore, if $s \geq qk + 2q - 2$, then $f \circ g$ cannot be a
$q$-almost embedding by Volovikov's theorem. We deduce that $f$
  is not a $q$-almost-embedding either, and
  Theorem~\ref{t:kuhnel-general} immediately follows. This deduction
  requires the following lemma (proven in Section~\ref{s:prelim}) as
  in general, a composition of a $q$-almost-embedding and an
  almost-embedding is not always a $q$-almost-embedding.

\begin{lemma}
  \label{l:compose_ae_general}
  Let $K$ and $L$ be simplicial complexes and $X$ a topological space.
  Suppose $g$ is an almost embedding of $K$ into $|L|$ and $f$ is a
  $q$-almost embedding of $L$ into $X$ for some integer $q \geq 2$.
  Then $f \circ g$ is a $q$-almost embedding of $K$ into $X$, provided
  that $g$ is the realization of a simplicial map $\gsimp$ from some
  subdivision $K'$ of $K$ to $L$.
\end{lemma}

\begin{remark}
 The third author proved in his thesis~\cite{patak:thesis-2015} a slightly
  better bound on $n_0$ in Lemma~\ref{l:chain_p}, namely
 $n_0=\binom{s}{k}b(s-2k) + s+1$.
 The proof, however, uses colorful version of Lemma~\ref{l:odd}. Since the proof of the colorful version
 is long and technical and in the end it only improves the bound in Theorem~\ref{t:kuhnel-2} by $2$,
 we have decided to present the more accessible version of the argument.
\end{remark}

\paragraph{Paper organization.}

Before we establish Lemma~\ref{l:chain_p} (in Section~\ref{s:strong}),
thus completing the proof of Theorem~\ref{t:kuhnel-general}, we first prove a
weaker version that introduces the main ideas in a simpler setting,
and yields a weaker bound for $n_0$, stated in
Equation~\eqref{eq:weak-bound}. The reader interested only in
  the case $q=2$ may want to consult a preliminary version of this
  paper~\cite{Kuhnel_conference_version} tailored to that case (where
  homology computations are without signs and the construction of the
  subdivision $D$ is simpler).

\section{Preliminaries}
\label{s:prelim}

We begin by fixing some terminology and notation. We will use
$\card(U)$ to denote the cardinality of a set $U$.

We recall that the \emph{stellar subdivision} of a maximal face
$\vartheta$ in a simplicial complex $K$ is obtained by removing
$\vartheta$ from $K$ and adding a cone $a_\vartheta \ast (\partial
\vartheta)$, where $a_\vartheta$ is a newly added vertex, the apex of
the cone (see Figure~\ref{f:stellar}).
\begin{figure}
\begin{center}
\includegraphics{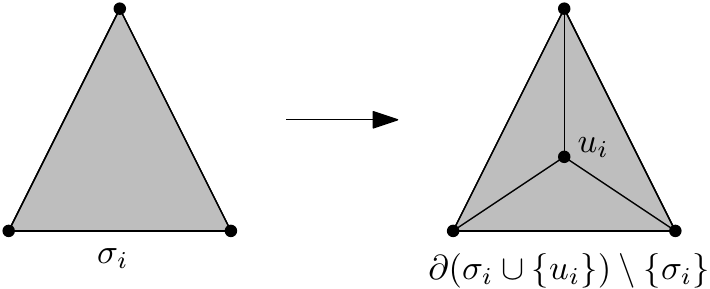}
\caption{A stellar subdivision of a simplex.}
\label{f:stellar}
\end{center}
\end{figure}

Throughout this paper we only work with homology groups and Betti numbers over
$\Z_p$, and for simplicity, we mostly drop the coefficient
group $\Z_p$ from the notation. Moreover, we will need to switch back and forth
between singular and simplicial homology. More precisely, if $K$ is a
simplicial complex then $H_*(K)$ will mean the simplicial homology of $K$,
whereas $H_*(X)$ will mean the singular homology of a topological space $X$. In
particular, $H_*(|K|)$ denotes the singular homology of the underlying space
$|K|$ of a complex $K$. We use analogous conventions for $C_*(K), C_*(X)$ and
$C_*(|K|)$ on the level of chains, and likewise for the subgroups of cycles and
boundaries, respectively.\footnote{We remark that throughout this paper, we
will only work with spaces that are either (underlying spaces of) simplicial
complexes or topological manifolds. Such spaces are homotopy equivalent to CW
complexes \cite[Corollary~1]{Milnor:On-spaces-having-the-homotopy-type-1959},
and so on the matter of homology, it does not really matter which (ordinary,
i.e., satisfying the dimension axiom) homology theory we use as they are all
naturally equivalent for CW complexes
\cite[Thm.~4.59]{Hatcher:AlgebraicTopology-2002}. However the distinction
between the simplicial and the singular setting will be relevant on the level
of chains.} Given a cycle $c$, we denote by $[c]$ the homology class it
represents.

A mapping $h \colon |K| \to X$ induces a chain map $h_\sharp^{\sing}
\colon C_*(|K|) \to C_*(X)$ on the level of singular chains;
see~\cite[Chapter 2.1]{Hatcher:AlgebraicTopology-2002}. There is also
a canonical chain map $\iota_K \colon C_*(K) \to C_*(|K|)$ inducing
the isomorphism of $H_*(K)$ and $H_*(|K|)$, see again~\cite[Chapter
2.1]{Hatcher:AlgebraicTopology-2002}. We define $h_\sharp \colon
C_*(K) \to C_*(X)$ as $h_\sharp := h_\sharp^{\sing} \circ \iota_K$.
The three chain maps mentioned above also induce maps $h_*^{\sing}$,
$(\iota_K)_*$, and $h_*$ on the level of homology satisfying $h_* =
h_*^{\sing} \circ (\iota_K)_*$.
We need a technical lemma saying that our maps compose, in a right
way, on the level of homology.

\begin{lemma}
  \label{l:commutative_diagram}
  Let $K$ and $L$ be simplicial complexes and $X$ a topological space.
  Let $j_{\simp}$ be a simplicial map from $K$ to $L$, $j\colon |K|
  \to |L|$ the continuous map induced by $j_{\simp}$ and $h\colon |L|
  \to X$ be another continuous map. Then $h_* \circ (j_{\simp})_* = (h
  \circ j)_*$ where $(j_{\simp})_*\colon H_*(K) \to H_*(L)$ is the map
  induced by $j_{\simp}$ on the level of simplicial homology and the maps $h_*$ and $(h \circ j)_*$ are as defined above.
\end{lemma}

\begin{proof}
The proof follows from the commutativity of the diagram below.

\includegraphics{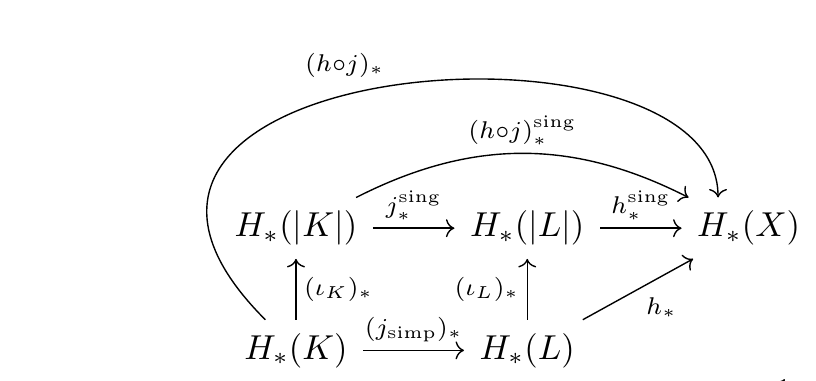}


\medskip 
 The commutativity of the lower right triangle follows from the definition of $h_*$.
 Similarly $(h \circ j)_* = (h \circ j)_*^{\sing} \circ (\iota_K)_*$. The fact that 
 $(h\circ j)_*^{\sing} =  h_*^{\sing} \circ j_*^{\sing}$
 follows from
 functionarility of the singular homology. The commutativity of the
 square follows from the naturality of the equivalence of the singular
 and simplicial homology; see~\cite[Thm
   34.4]{Munkres:AlgebraicTopology-1984}.
 \end{proof}

We now prove the final technical step of our approach, stated in the introduction.

\begin{proof}[Proof of Lemma~\ref{l:compose_ae_general}]
  Let $\sigma_1, \dots, \sigma_q$ be $q$ pairwise disjoint faces of
  $K$. Our task is to show $f\circ g(|\sigma_1|)\cap \cdots \cap f
  \circ g(|\sigma_q|) = \emptyset$. Let $\vartheta_i$ be a face of
  $K'$ that subdivides $\sigma_i$ for $i \in [q]$. We are done, if we
  prove
  \begin{equation}
    \label{e:thetas}
    f\circ g(|\vartheta_1|)\cap \cdots \cap f
    \circ g(|\vartheta_q|) = \emptyset
  \end{equation}
  for every such possible choice of $\vartheta_1, \dots, \vartheta_q$.

  The faces $\vartheta_1, \dots, \vartheta_q$ are pairwise disjoint
  since $\sigma_1, \dots, \sigma_q$ are pairwise disjoint. Since
  $\gsimp$ is a simplicial map inducing an almost embedding, the faces
  $\gsimp(\vartheta_1), \dots, \gsimp(\vartheta_q)$ are pairwise
  disjoint faces of $L$. Consequently, \eqref{e:thetas} follows from
  the fact that $f$ is a $q$-almost embedding.
\end{proof}

\section{Proof of Lemma~\ref{l:chain_p} with a weaker bound on $n_0$}\label{s:weak}

Let $k,b,s$ be fixed integers. We consider a $2k$-manifold $M$ with
$k$th Betti number $b$, a map $f:|\skelsim{k}{n}| \to M$.
Recall that although we want to build an almost-embedding, homology is computed over~$\Z_p$. The strategy of
  our proof of Lemma~\ref{l:chain_p} is to start by designing an
  auxiliary chain map
$$
\varphi\colon C_*\left(\skelsim{k}{s}\right)\to C_*\left(\skelsim{k}{n}\right).
$$
that behaves as an almost-embedding, in the sense that whenever
$\sigma$ and $\sigma'$ are disjoint $k$-faces of $\simplex s$,
$\varphi(\sigma)$ and $\varphi(\tau)$ have disjoint supports, and such
that for every $(k+1)$-face $\tau$ of $\simplex{s}$ the homology
class $[(f_\sharp \circ \varphi)(\partial \tau)]$ is trivial. We then
use $\varphi$ to design a subdivision $D$ of $\skelsim ks$ and a
simplicial map $\gsimp: D \to\skelsim kn$ that induces a map $g: |D|
\to |\skelsim kn|$ with the desired properties: $g$ is an
almost-embedding and $(f \circ g)_*([\partial \tau])$ is trivial for
all $(k+1)$-faces $\tau$ of $\simplex{s}$. Since the cycles $\partial
\tau$, for $(k+1)$-faces $\tau$ of $\simplex s$, generate all
$k$-cycles of $\skelsim ks$, this implies that $(f \circ g)_*$ is
trivial.

The purpose of this section is to give a first implementation of the
above strategy that proves Lemma~\ref{l:chain_p} with a bound of
\begin{equation}
\label{eq:weak-bound}
\hfill
n_0 \ge  \pth{\binom{s+1}{k+1}-1}p^{b\binom{s+1}{k+1}} + s + 1.
\hfill
\end{equation}
In Section~\ref{s:strong} we then improve this bound to
$\binom{s}{k}b(s-2k) + 2s-2k+1$ at the cost of some technical
complications (note that the improved bound is independent of $p$).

\bigskip

Throughout the rest of this paper we use the following notations. We
let $\{v_1, v_2, \ldots, v_{n+1}\}$ denote the set of vertices of
$\simplex{n}$ and we assume that $\simplex{s}$ is the induced
subcomplex of $\simplex{n}$ on $\{v_1, v_2, \ldots, v_{s+1}\}$. We let
$U = \{v_{s+2}, v_{s+3}, \ldots, v_{n+1}\}$ denote the set of vertices
of $\simplex{n}$ \emph{unused} by $\simplex{s}$.  We let
$m=\binom{s+1}{k+1}$ and denote by $\sigma_1, \sigma_2,\ldots,
\sigma_m$ the $k$-faces of $\simplex s$, ordered lexicographically.

Later on, when working with homology, we compute the simplicial
homology with respect to this fixed order on the vertices of
$\simplex{n}$. In particular, the boundary of a $j$-simplex $\vartheta =
\{v_{i_1}, v_{i_2},\dots, v_{i_{j+1}}\}$ where $i_1 \leq i_2 \leq \dots \leq i_{j+1}$ is
$$
\partial \vartheta = \sum\limits_{\ell=1}^{j+1} (-1)^{\ell+1} \vartheta \setminus
\{v_{i_{\ell}}\}.
$$

\subsection{Construction of $\varphi$}

For every face $\vartheta$ of $\simplex s$ of dimension at most $k-1$
we set $\varphi(\vartheta) = \vartheta$. We then ``route''
each~$\sigma_i$ by mapping it to its stellar subdivision with an apex $u
\in U$, \emph{i.e.} by setting $\varphi(\sigma_i)$ to $\sigma_i +
(-1)^k z(\sigma_i,u)$ where $z(\sigma_i,u)$ denotes the cycle $\partial(\sigma_i \cup
\{u\})$; see Figure~\ref{f:phi_edge} for the case $k = 1$.

\begin{figure}
\begin{center}
\includegraphics{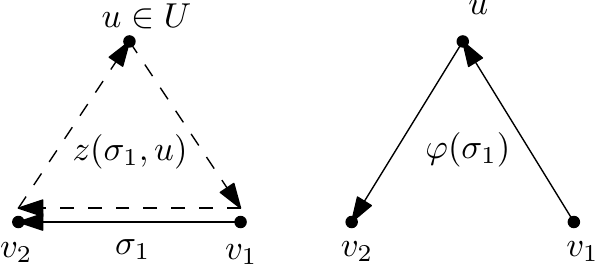}
\caption{Rerouting $\sigma_i$ for $k=1$. The support
  of $z(\sigma_i,u)$ is dashed on the left, and the support of
  resulting $\varphi(\sigma_i)$ is on the right.}
\label{f:phi_edge}
\end{center}
\end{figure}

We ensure that $\varphi$ behaves as an almost-embedding by using a
different apex $u\in U$ for each $\sigma_i$. The difficulty is to
choose these $m$ apices in a way that $[f_\sharp(\varphi(\partial
\tau))]$ is trivial for every $(k+1)$-face $\tau$ of $\simplex{s}$.
To that end we associate to each $u\in U$ the sequence
$$ 
\vv(u) := ([f_\sharp(z(\sigma_1,u))],
  [f_\sharp(z(\sigma_2,u))],\dots,
[f_\sharp(z(\sigma_m,u)]) \in H_k(M)^m,
$$
and we denote by $\vv_i(u)$ the $i$th element of $\vv(u)$.  We work
with $\Z_p$-homology, so $H_k(M)^m$ is finite; more precisely, its
cardinality equals $p^{bm}$. From $n \ge n_0 = (m-1)p^{bm} + s + 1$ we
get that $\card(U) \ge (m-1)\card(H_k(M)^m)+1$.
 The pigeonhole
principle then guarantees that there exist $m$ distinct vertices
$u_1, u_2, \ldots, u_m$ of $U$ such that $\vv(u_1) = \vv(u_2) = \cdots
= \vv(u_m)$. We use $u_i$ to ``route'' $\sigma_i$ and put
\begin{equation}\label{eq:defphi}
\hfill
\varphi(\sigma_i):= \sigma_i + (-1)^kz(\sigma_i,u_i).
\hfill
\end{equation}
We finally extend $\varphi$ linearly to $C_*\left(\skelsim{k}{s}\right)$.

\begin{lemma}\label{l:zero_homology_points}
The map  $\varphi$ is a chain map and $\bigl[f_\sharp\bigl(\varphi(\partial
  \tau)\bigr)\bigr]= 0$ for every $(k+1)$-face $\tau\in\simplex{s}$.
\end{lemma}

Before proving the lemma, we establish a simple claim that will also
be useful later.

\begin{claim}
\label{c:boundary_tau_points}
  Let $\tau$ be a $(k+1)$-face of $\simplex s$ and let $u \in U$. Let
  $\sigma_{i_1}, \dots, \sigma_{i_{k+2}}$ be all the $k$-faces of $\tau$
  sorted lexicographically, that is, $i_1 \leq \cdots \leq i_{k+2}$. Then
  \begin{equation}
    \label{e:boundary_tau_points}
      \partial \tau =
       z(\sigma_{i_1}, u) - z(\sigma_{i_2}, u) + \cdots +
   (-1)^{k+1}z(\sigma_{i_{k+2}},u).
   \end{equation}
\end{claim}
	    
\begin{proof}
This follows from expanding the equation $0 = \partial^2(\tau\cup\{u\})$.
Indeed, \[ \begin{split} 0 = \partial^2(\tau\cup\{u\}) &= \partial\bigl(
\sigma_{i_{k+2}} \cup \{u\} - \sigma_{i_{k+1}} \cup \{u\} + \cdots +
(-1)^{k+1}\sigma_{i_1} \cup \{u\} + (-1)^{k+2}\tau\bigr) \\ &= (-1)^{k+1}\bigl(
- \partial \tau + z(\sigma_{i_1},u) - z(\sigma_{i_2},u) + \cdots +(-1)^{k+1}
z(\sigma_{i_{k+2}},u) \bigr).\\ \end{split} \] 
\end{proof}

\begin{proof}[Proof of Lemma~\ref{l:zero_homology_points}]
  The map $\varphi$ is the identity on $\ell$-chains with $\ell \leq
  k-1$ and Equation~\eqref{eq:defphi} immediately implies that
  $\partial \varphi(\sigma) = \partial \sigma$ for every $k$-simplex
  $\sigma$. It follows that $\varphi$ is a chain map.

  Now let $\tau$ be a $(k+1)$-simplex of $\simplex{s}$ and let
  $\sigma_{i_1}, \dots, \sigma_{i_{k+2}}$ be its $k$-faces. We have
\[\begin{aligned}
  f_\sharp \circ \varphi (\partial\tau) = f_\sharp \circ \varphi \left(
    \sum_{j=1}^{k+2} (-1)^{k+j} \sigma_{i_j} \right) & = f_\sharp \left( \sum_{j=1}^{k+2} (-1)^{j+k}\left(\sigma_{i_j} +
      (-1)^kz(\sigma_{i_j},u_{i_j})\right) \right)\\
  & = f_\sharp(\partial\tau) +
  \sum_{j=1}^{k+2} (-1)^jf_\sharp\bigl(z(\sigma_{i_j}, u_{i_j})\bigr).
\end{aligned}\]
  $\bigl[f_\sharp\bigl(z(\sigma_{i_j},u_{\ell})
  \bigr)\bigr] = \vv_{i_j}(u_\ell)$ is independent of the value
  $\ell$. When passing to the homology classes in the above identity,
  we can therefore replace each $u_{i_j}$ with $u_1$, and obtain,
  $$
    \left[f_\sharp\circ \varphi (\partial\tau)\right] =
    [f_\sharp(\partial\tau)] + \sum_{j=1}^{k+2} (-1)^j
    \Bigl[f_\sharp\bigl(z(\sigma_{i_j}, u_1)
    \bigr)\Bigr]
    = 
    \Bigl[f_\sharp\Bigl(
      \partial \tau + \sum_{j=1}^{k+2} (-1)^j z(\sigma_{i_j}, u_1)
    \Bigr)\Bigr].
  $$
  This class is trivial by Claim~\ref{c:boundary_tau_points}.
  Figure~\ref{f:homology_trivial} illustrates the geometric
    intuition behind this proof.
\end{proof}

\begin{figure}{t}
  \begin{center}
    \includegraphics{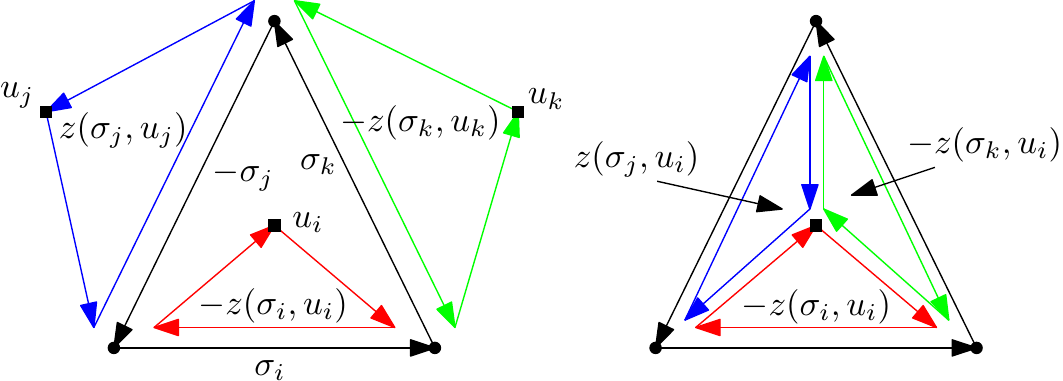}
    \caption{The geometric intuition behind the proof of
      Lemma~\ref{l:zero_homology_points}, for $k = 1$ and $u_{i_1}=
      u_1$ (cycles of same color are in the same homology class; the
      class on the right is trivial, because the edges cancel out in
      pairs).\label{f:homology_trivial}}
\end{center}
\end{figure}

\subsection{Subdivisions and orientations}
\label{ss:subdivisions}

Our next task is the construction of $D$ and $g$; however, we first mention
a few properties of subdivisions.

Let us consider a simplicial complex $K$ and a subdivision $S$ of
$K$. (So $K$ and $S$ are regarded as geometric simplicial
  complexes, and for every simplex $\eta$ of $S$ there is a simplex
  $\vartheta$ of $K$ such that $\eta \subseteq \vartheta$. In this case, we say
  that $\eta$ \emph{subdivides} $\vartheta$.)  There is a
canonical chain map $\rho\colon C_*(K) \to C_*(S)$ that induces an
isomorphism in homology. Intuitively, $\rho$ maps a simplex
$\vartheta$ of $K$ to a sum of simplices of $S$ of the same dimension
that subdivide $\vartheta$. However, we have to be careful about the
$\pm 1$ coefficients in the sum.

We work with the ordered simplicial homology, that is, we order
  the vertices of $K$ as well as the vertices of $S$.
  We want to define the mutual orientation $\Or(\eta, \vartheta)
    \in \{-1,1\}$ of a $j$-simplex $\vartheta$ of $K$ and a
    $j$-simplex~$\eta$ of $S$ that subdivides $\vartheta$. We set up
    $\Or(\eta, \vartheta)$ to be $1$ if the orientations of
    $\vartheta$ and $\eta$ agree, and $-1$ if they disagree; the
    orientation of each geometric simplex is computed relative to the
    order of its vertices in $K$ or $S$ (with respect to a fixed base
    of their common affine hull, say).
  Then we set
\begin{equation}
  \label{e:rho}
  \rho(\vartheta) = \sum\limits_{\eta} \Or(\eta,\vartheta) \eta
\end{equation}
where the sum is over all simplices $\eta$ in $S$ of the same
dimension as $\vartheta$ which subdivide $\vartheta$. Finally, we
extend $\rho$ to a chain map. It is routine to check that $\rho$
commutes with the boundary operator and that it induces an isomorphism
on homology.
It is also useful to describe $\rho$ in the specific case where $S$ is
a stellar subdivision of a complex $K$ consisting of a single
$k$-simplex. Here, we assume that $w_1, \dots, w_{j+1}$ are the
vertices of $K$ in this order (in $K$ as well as in $S$) and $a$ is
the apex of $S$, which comes last in the order on $S$. We also
consider $S$ as a subcomplex of the $(k+1)$-simplex on $w_1, \dots,
w_{k+1}, a$. And we use the notation $z(\vartheta, a) =
\partial(\vartheta \cup \{a\})$, analogously as previously in the case
of $k$-faces of $\Delta_s$.

\begin{lemma} 
  \label{l:stellar_rho}  
  In the setting above, let $\vartheta$ be the $k$-face of $K$. Then
  $\rho(\vartheta) = \vartheta + (-1)^k z(\vartheta, a)$.
\end{lemma}

\begin{proof}
  Let $\eta_i := \vartheta \cup \{a\} \setminus \{w_i\}$ for $i \in [k+1]$.
  Then $\eta_i$ are all faces of $S$ subdividing $\vartheta$. We have
  $\Or(\eta_i,\vartheta) = (-1)^{i+k +1}$ as $\eta_i$ has the same orientation
  as $\vartheta$ with respect to a modified order of vertices of $\vartheta$
  obtained by replacing $w_i$ with $a$. Therefore $\rho(\vartheta) =
  \sum\limits_{i=1}^{k+1} (-1)^{i+k+1} \eta_i$. 
  On the other hand,
$$
z(\vartheta, a) = \partial(\vartheta \cup \{a\}) = \left(\sum\limits_{i=1}^{k+1}(-1)^{i+1}
\eta_i \right) +
(-1)^{k+3} \vartheta =(-1)^k(\rho(\vartheta)-\vartheta).
$$
\end{proof}
 
\subsection{Construction of $D$ and $g$}

The definition of $\varphi$, an in particular
Equation~\eqref{eq:defphi}, suggests to construct our subdivision $D$
of $\skelsim ks$ by simply replacing every $k$-face of $\skelsim ks$
by its stellar subdivision.  Let $a_i$ denote the new vertex
introduced when subdividing $\sigma_i$. We fix a linear order on vertices
of $D$ in such a way that we reuse the order of vertices that also belong to
$\skelsim ks$ and then the vertices $a_i$ follow in arbitrary order.

We define a simplicial map $\gsimp \colon D \to \skelsim kn$ by
putting $\gsimp(v) = v$ for every original vertex $v$ of $\skelsim
ks$, and $\gsimp(a_i) = u_i$ for $i \in [m]$. This $\gsimp$ induces a
map $g \colon |\skelsim ks| \to |\skelsim kn|$ on the geometric
realizations. Since the $u_i$'s are pairwise distinct, $g$ is an
embedding\footnote{We use the full strength of almost-embeddings when
  proving Lemma~\ref{l:chain_p} with the better bound on $n_0$.}, so
Condition~1 of Lemma~\ref{l:chain_p} holds.

In principle, we would like to derive Condition~2 of
Lemma~\ref{l:chain_p} by observing that $g$ `induces' a
chain map from $C_*(\skelsim ks)$ to $C_*( \skelsim kn)$ that
coincides with $\varphi$. Making this a formal statement is thorny
because $g$, as a continuous map, naturally induces a chain map
$g_\sharp$ on singular rather than simplicial chains. We can't use
directly $\gsimp$ either, since we are interested in a map from
$C_*(\skelsim ks)$ and not from~$C_*(D)$.

We handle this technicality as follows. We consider the chain map $\rho \colon
C_*(\skelsim ks) \to C_*(D)$ from \eqref{e:rho}. 
This map induces an isomorphism~$\rho_*$ in
homology. In addition $\varphi = (\gsimp)_\sharp \circ \rho$ where
$(\gsimp)_\sharp \colon C_*(D) \to C_*(\skelsim kn)$ denotes the
(simplicial) chain map induced by $\gsimp$. Indeed, all three maps are
the identity on simplices of dimension at most $k - 1$. For a $k$-simplex
$\sigma$, the map $\gsimp$ is an order preserving isomorphism when restricted
to the subdivision of $\sigma$ (in $D$). Therefore, the required equality
$\varphi(\sigma) = (\gsimp)_\sharp \circ \rho(\sigma)$ follows
from~\eqref{eq:defphi} and Lemma~\ref{l:stellar_rho}.

We thus have in homology
$$ f_* \circ \varphi_* = f_* \circ (\gsimp)_* \circ \rho_*  $$
and since $\rho_*$ is an isomorphism and $f_* \circ \varphi_*$ is
trivial by Lemma~\ref{l:zero_homology_points}, it follows that $f_* \circ
(\gsimp)_*$ is also trivial. Since $f_* \circ (\gsimp)_* = (f \circ
g)_*$ by Lemma~\ref{l:commutative_diagram}, $(f \circ g)_*$ is trivial
as well. This concludes the proof of Lemma~\ref{l:chain_p} with the
weaker bound.

\section{Proof of Lemma~\ref{l:chain_p}
}\label{s:strong}

We now prove Lemma~\ref{l:chain_p} with the bound claimed in the
statement, namely
$$ n_0 = \binom{s}{k}b(s-2k) + 2s - 2k + 1.$$
Let $k,b,s$ be fixed integers. We consider a $2k$-manifold $M$ with
$k$th $\Z_p$-Betti number $b$, a map $f:|\skelsim{k}{n}| \to M$, and we
assume that $n \ge n_0$.

The proof follows the same strategy as in Section~\ref{s:weak}: we
construct a chain map $\varphi\colon C_*(\skelsim ks) \to C_*(\skelsim
kn)$ such that the homology class $[(f_\sharp \circ \varphi)(\partial
\tau)]$ is trivial for all $(k+1)$-faces $\tau$ of $\Delta_s$, then
upgrade $\varphi$ to a continuous map $g\colon |\skelsim ks| \to
|\skelsim kn|$ with the desired properties.

When constructing $\varphi$, we refine the arguments of
Section~\ref{s:weak} to ``route'' each $k$-face using not only one,
but several vertices from $U$; this makes finding ``collisions''
easier, as we can use linear algebra arguments instead of the
pigeonhole principle. This comes at the cost that when upgrading $g$,
we must content ourselves with proving that it is an almost-embedding.
This is sufficient for our purpose and has an additional benefit: the
same group of vertices from $U$ may serve to route several $k$-faces
provided they pairwise intersect in $\skelsim sk$.

\subsection{Construction of $\varphi$}

We use the same notation regarding $v_1, \ldots, v_{n+1}$, $\simplex{n}$,
$\simplex{s}$, $U$, $m=\binom{s+1}{k+1}$ and $\sigma_1, \sigma_2,\ldots,
\sigma_m$ as in Section~\ref{s:weak}.

\paragraph{Multipoints and the map $\vv$.}

As we said we plan to route $k$-faces of $\simplex{s}$ through certain collections of
vertices from $U$ (weighted); we will call these collections multipoints. 
It is more convenient to work with them on the level of formal linear
combinations.
Let $C_0(U)$ denote the $\Z_p$-vector space of formal linear combinations of
vertices from $U$. A
\emph{multipoint} is an element of $C_0(U)$ whose 
 coefficients sum
to $1$ (in $\Z_p$, of course).
The multipoints form an affine subspace of $C_0(U)$ which we
denote by $\M$. The \emph{support}, $\sup(\mu)$, of a multipoint $\mu \in \M$ is the
set of vertices $v \in U$ with non-zero coefficient in $\mu$. We say
that two multipoints are \emph{disjoint} if their supports are
disjoint.

For any $k$-face $\sigma_i$ and any multipoint $\mu = \sum_{u \in U}\lambda_uu$ 
we define:
$$ z(\sigma_i, \mu):=\sum_{u \in \sup(\mu)} \lambda_u z(\sigma_i,
u) := \sum_{u \in \sup(\mu)} \lambda_u \partial(\sigma_i \cup \{u\}).$$
Now, we proceed as in Section~\ref{s:weak} but replace unused
points by multipoints of $\M$ and the cycles $z(\sigma_i, u)$ with
the cycles $z(\sigma_i, \mu)$. Since $\Z_p$ is a field, $H_k(M)^m$ is
a vector space and we can replace the sequences $\vv(u)$ of
Section~\ref{s:weak} by the linear map
$$\vv: \left\{\begin{array}{rcl} \M & \to & H_k(M)^m\\
    \mu & \mapsto & ([f_\sharp(z(\sigma_1,\mu))], [f_\sharp(z(\sigma_2,\mu))],\ldots,
    [f_\sharp(z(\sigma_m,\mu))])\end{array}\right.$$

\paragraph{Finding collisions.}

The following lemma takes advantage of the vector space structure of
$H_k(M)^m$ and the affine structure of $\M$ to find disjoint
multipoints $\mu_1, \mu_2, \ldots$ to route the $\sigma_i$'s more
effectively than by simple pigeonhole.

\begin{lemma}\label{l:odd}
  For any $r \ge 1$, any $\Z_p$-vector space $V$, and any affine map
  $\psi\colon \M \to V$, if $\card(U) \ge
  (\dim(\psi(\M))+1)(r-1)+1$ then $\M$ contains $r$ disjoint
  multipoints $\mu_1, \mu_2, \ldots, \mu_r$ such that $\psi(\mu_1) =
  \psi(\mu_2) = \cdots = \psi(\mu_r)$.
\end{lemma}
\begin{proof}
  Let us write $U = \{v_{s+2}, v_{s+3}, \ldots, v_{n+1}\}$ and
  $d=\dim(\psi(\M))$.  We first prove by induction on $r$ the
  following statement:
  \begin{quote}
    If $\card(U) \ge (d+1)(r-1)+1$ there exist $r$ pairwise disjoint
    subsets $I_1, I_2, \ldots,I_r \subseteq U$ whose image under
    $\psi$ have affine hulls with non-empty intersection.
  \end{quote}
  (This is, in a sense, a simple affine version of Tverberg's
  theorem.)  The statement is obvious for $r = 1$, so assume that $r
  \geq 2$ and that the statement holds for $r-1$. Let $A$ denote the
  affine hull of $\{\psi(v_{s+2}), \psi(v_{s+3}),\dots, \psi(v_{n+1})\}$
  and let $I_r$ denote a minimal cardinality subset of $U$ such that
  the affine hull of $\{\psi(v): v \in I_r\}$ equals $A$. Since $\dim
  A \le d$ the set $I_r$ has cardinality at most $d+1$. The cardinality
  of $U \setminus I_r$ is at least $(d+1)(r-2)+1$ so we can apply the
  induction hypothesis for $r-1$ to $U \setminus I_r$. We thus obtain
  $r-1$ disjoint subsets $I_1, I_2, \ldots, I_{r-1}$ whose images
  under $\psi$ have affine hulls with non-empty intersection. Since
  the affine hull of $\psi(U \setminus I_r)$ is contained in the
  affine hull of $\psi(I_r)$, the claim follows.

  Now, let $a \in V$ be a point common to the affine hulls of
  $\psi(I_1), \psi(I_2), \ldots,\psi(I_r)$. Writing $a$ as an affine
  combination in each of these spaces, we get
  $$ a = \sum_{u \in J_1}\lambda^{(1)}_u\psi(u) = \sum_{u \in
  J_2}\lambda^{(2)}_u\psi(u) = \cdots = \sum_{u \in
  J_r}\lambda^{(r)}_u\psi(u)$$
  where $J_j\subseteq I_j$ and $\sum_{u \in J_j}\lambda^{(j)}_u =1$ for any $j \in [r]$.  
  Setting $\mu_j = \sum_{u
  \in J_j} \lambda_u^{(j)} u$ finishes the proof.
\end{proof}

\paragraph{Computing the dimension of $\vv(\M)$.}

Having in mind to apply Lemma~\ref{l:odd} with $V = H_k(M)^m$ and
$\psi = \vv$, we now need to bound from above the dimension of
$\vv(\M)$. An obvious upper bound is $\dim H_k(M)^m$, which equals $bm
= b\binom{s+1}{k+1}$. A better bound can be obtained by an argument
analogous to the proof of Lemma~\ref{l:zero_homology_points}. We first
extend Claim~\ref{c:boundary_tau_points} to multipoints.
\begin{claim}
\label{c:boundary_tau_multipoints}
  Let $\tau$ be a $(k+1)$-face of $\simplex s$ and let $\mu \in \M$. Let
  $\sigma_{i_1}, \dots, \sigma_{i_{k+2}}$ be all the $k$-faces of $\tau$
  sorted lexicographically. Then
  \begin{equation}
 \label{e:boundary_tau_multipoints}
      \partial \tau =
       z(\sigma_{i_1}, \mu) - z(\sigma_{i_2}, \mu) + \cdots +
   (-1)^{k+1}z(\sigma_{i_{k+2}},\mu).
   \end{equation}
\end{claim}

\begin{proof}
  By Claim~\ref{c:boundary_tau_points} we know that
  \eqref{e:boundary_tau_multipoints} is true for points. For a multipoint
  $\mu = \sum_{u \in U}\lambda_u u$, we get \eqref{e:boundary_tau_multipoints} as a linear combination of
  equations~\eqref{e:boundary_tau_points} for the points $u$ with the `weight' $\lambda_u$ (note that
  $\sum_{u \in U} \lambda_u = 1$; therefore the corresponding combination of
  the left-hand sides of~\eqref{e:boundary_tau_points} equals $\partial
  \tau$).
\end{proof}

\begin{lemma}
  \label{l:im_vk}
  $\dim (\vv(\M)) \le b\binom sk$.
\end{lemma}
\begin{proof}
  Let $\tau$ be a $(k+1)$-face of $\Delta_s$ and let $\sigma_{i_1},
  \dots, \sigma_{i_{k+2}}$ denote its $k$-faces. 
  For any multipoint $\mu$, Claim~\ref{c:boundary_tau_multipoints} implies  
  $$ [f_\sharp(\partial\tau)] = \sum_{j=1}^{k+2} (-1)^{j+1} [f_\sharp(z(\sigma_i,
  \mu))] = \sum_{j=1}^{k+2} (-1)^{j+1}\vv_{i_j}(\mu);$$ 
therefore
$$\qquad \vv_{i_{k+2}}(\mu) = (-1)^{k+1}[f_\sharp(\partial\tau)] +
\sum_{j=1}^{k+1} (-1)^{j+k+1}\vv_{i_j}(\mu).$$
  Each vector
  $\vv(\mu)$ is thus determined by the values of the $\vv_{j}(\mu)$'s
  where $\sigma_j$ contains the vertex $v_1$. Indeed, the vectors
  $[f_\sharp(\partial\tau)]$ are independent of $\mu$, and for any
  $\sigma_i$ not containing $v_1$ we can eliminate $\vv_i(\mu)$ by
  considering $\tau := \sigma_i \cup \{v_1\}$ (and setting $\sigma_{i_{k+2}} =
  \sigma_i$). For each of the
  $\binom{s}{k}$ faces $\sigma_j$ that contain $v_1$, the vector
  $\vv_{j}(\mu)$ takes values in $H_k(M)$ which has dimension at most
  $b$. It follows that $\dim \vv(\M) \leq b\binom{s}{k}$.
\end{proof}

\paragraph{Coloring graphs to reduce the number of multipoints used.}

We could now apply Lemma~\ref{l:odd} with $r=m$ to obtain one
multipoint per $k$-face, all pairwise disjoint, to proceed with our
``routing''. As mentioned above, however, we only need that $\varphi$
is an almost-embedding, so we can use the same multipoint for several
$k$-faces provided they pairwise intersect. Optimizing the number of
multipoints used reformulates as the following graph coloring
problem:

\begin{quote}
  Assign to each $k$-face $\sigma_i$ of $\Delta_s$ some color $c(i)
  \in \N$ such that $\card\{c(i): 1 \le i \le m\}$ is minimal and
  disjoint faces use distinct colors.
\end{quote}

\noindent
This question is classically known as Kneser's graph coloring
problem and an optimal solution uses $s-2k+1$ colors~\cite{lovasz78,
  Matousek:BorsukUlam-2003}. Let us spell out one such coloring
(proving its optimality is considerably more difficult, but we do not
need to know that it is optimal). For every $k$-face $\sigma_i$ we let
$\min \sigma_i$ denote the smallest index of a vertex in $\sigma_i$.
When $\min \sigma_i \leq s - 2k$ we set $c(i) = \min \sigma_i$,
otherwise we set $c(i) = s - 2k + 1$.  Observe that any $k$-face with
color $c \leq s-2k$ contains vertex $v_c$.  Moreover, the $k$-faces
with color $s - 2k +1$ consist of $k+1$ vertices each, all from a set
of $2k+1$ vertices. It follows that any two $k$-faces with the same
color have some vertex in common.

\paragraph{Defining $\varphi$.}

We are finally ready to define the chain map $\varphi\colon
C_*(\skelsim{k}{s}) \to C_*(\skelsim{k}{n})$. Recall that we assume
that $n \ge n_0 = (\binom{s}{k}b+1)(r-1) + s +1$. Using the bound of
Lemma~\ref{l:im_vk} we can apply Lemma~\ref{l:odd} with $r = s - 2k +
1$, obtaining $s-2k+1$ multipoints $\mu_1, \mu_2, \ldots, \mu_{s-2k+1}
\in \M$. We set $\varphi(\vartheta) = \vartheta$ for any face
$\vartheta$ of $\simplex{s}$ of dimension less than $k$. We then
``route'' each $k$-face $\sigma_i$ through the multipoint $\mu_{c(i)}$
by putting
\begin{equation}
  \label{e:varphi} 
  \hfill
  \varphi(\sigma_i):= \sigma_i + (-1)^k z(\sigma_i,\mu_{c(i)}),
\hfill
\end{equation}
where $c(i)$ is the color of $\sigma_i$ in the coloring of the Kneser
graph proposed above. We finally extend $\varphi$ linearly to
$C_*(\Delta_s)$.

We need the following analogue of Lemma~\ref{l:zero_homology_points}.

\begin{lemma}\label{l:zero_homology_multipoints}
The map  $\varphi$ is a chain map and $\bigl[f_\sharp\bigl(\varphi(\partial
  \tau)\bigr)\bigr]= 0$ for every $(k+1)$-face
  $\tau\in\simplex{s}$.
\end{lemma}

The proof of Lemma~\ref{l:zero_homology_multipoints} is very similar to the
proof of Lemma~\ref{l:zero_homology_points}; it just replaces points with
multipoints and Claim~\ref{c:boundary_tau_points} with
Claim~\ref{c:boundary_tau_multipoints}. We therefore omit the proof.
We next argue that $\varphi$ behaves like an almost embedding.

\begin{lemma}
  \label{l:chain_ae}
  For any two disjoint faces $\vartheta, \eta$ of $\skelsim{k}{s}$,
  the supports of $\varphi(\vartheta)$ and $\varphi(\eta)$ use disjoint
  sets of vertices.
\end{lemma}
\begin{proof}
  Since $\varphi$ is the identity on chains of dimension at most
  $(k-1)$, the statement follows if neither face has dimension $k$.
  For any $k$-chain $\sigma_i$, the support of $\varphi(\sigma_i)$
  uses only vertices from $\sigma_i$ and from the support of
  $\mu_{c(i)}$. Since each $\mu_{c(i)}$ has support in $U$, which
  contains no vertex of $\simplex s$, the statement also holds when
  exactly one of $\vartheta$ or $\eta$ has dimension $k$. When both
  $\vartheta$ and $\eta$ are $k$-faces, their disjointness implies
  that they use distinct $\mu_j$'s, and the statement follows from the
  fact that distinct $\mu_j$'s have disjoint supports.
\end{proof}

\subsection{Construction of $D$ and $g$}

We define $D$ and $g$ similarly as in Section~\ref{s:weak}, but the
switch from points to multipoints requires to replace stellar
subdivisions by a slightly more complicated decomposition.

\begin{figure}
  \begin{center}
    \includegraphics{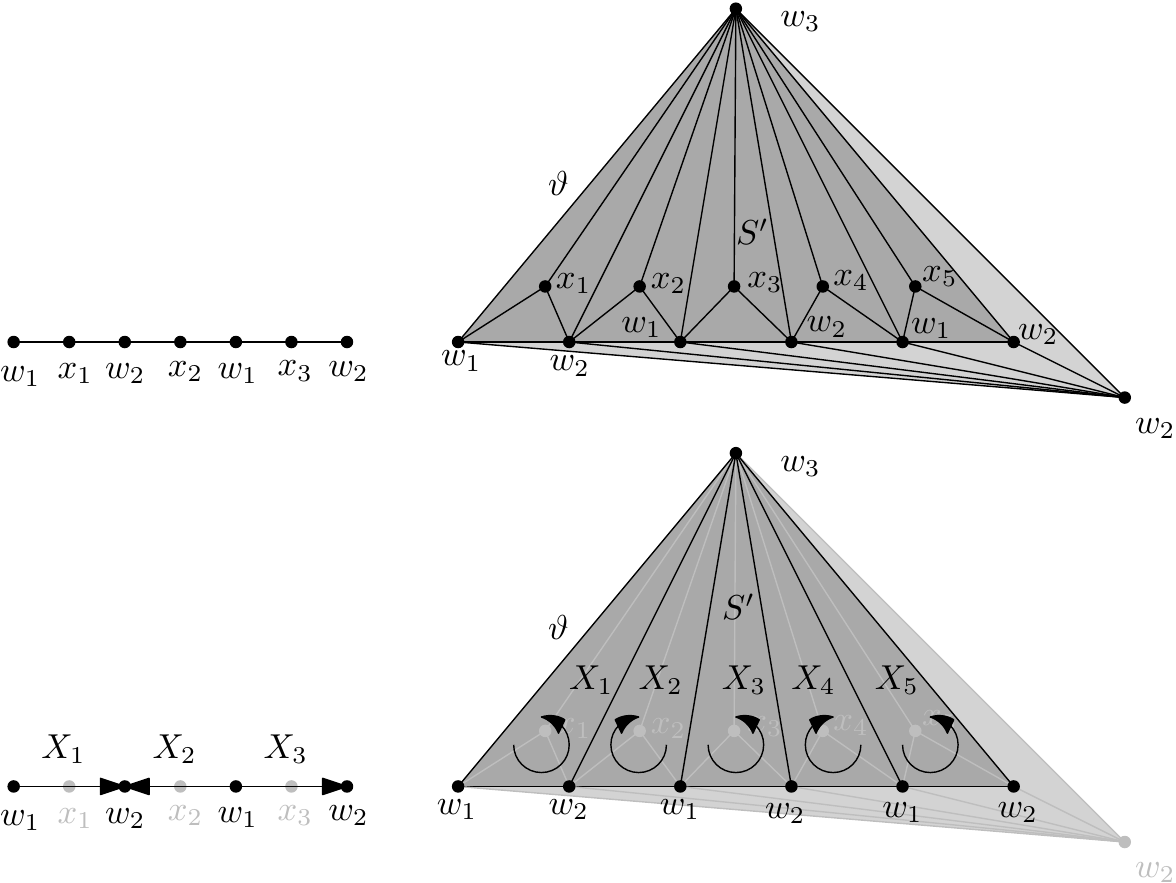}
  \end{center}
  \caption{Examples of subdivisions for $k=1$ and $\ell=3$ (left) and
    for $k=2$ and $\ell=5$ (right). The bottom pictures show the orientations
    of $|X_i|$ in the given ordering.
  \label{f:subdivide}}
\end{figure}

\paragraph{The subdivision $D$.}
We define $D$ so that it coincides with $\simplex s$ on the faces of
dimension at most $(k-1)$ and decomposes each face of dimension $k$
independently. The precise subdivision of a $k$-face $\sigma_i$
depends on the cardinality of the support of the multipoint $\mu_{c(i)}$
used to ``route'' $\sigma_i$ under $\varphi$, but the method is
generic and spelled out in the next lemma; refer to
Figure~\ref{f:subdivide}.

\begin{lemma}
  \label{l:subdivide}
  Let $k \geq 1$ and $\sigma = \{w_1, w_2, \ldots, w_{k+1}\}$ be a
  $k$-simplex. For any positive odd integer $\ell \ge 1$ there exist a
  subdivision $S$ of $\sigma$ in which no face of dimension $k-1$ or
  less is subdivided, and a labelling of the vertices of $S$ by
  $\{w_1, w_2, \dots, w_{k+1}, x_1, x_2, \dots, x_\ell\}$ (some labels may
  appear several times) satisfying the following properties.

  \begin{enumerate}
  \item Every vertex in $S$ corresponding to an original vertex $w_i$
    of $\sigma$ is labelled by $w_i$.
  \item No $k$-face of $S$ has its vertices labelled by $w_1, w_2,
    \ldots, w_{k+1}$,
  \item for every $j \in [\ell]$, the subdivision $S$ contains exactly one vertex
    labelled by $x_j$; this vertex appears in a copy $X_j$ of a
    stellar subdivision of a simplex labelled by $w_1, \ldots, w_{k+1}$ with the apex
    labelled $x_j$.
  \item Let us equip vertices of $S$ with a linear order which respects the
    order $w_1 \leq w_2 \leq \cdots \leq w_{k+1} \leq x_1 \leq \cdots \leq
    x_\ell$ of the labels. For each $j \in [\ell]$ considering $|X_j|$ as a simplex  
    in $|S| = |\sigma|$, such $|X_j|$ is oriented coherently with $|\sigma|$
    (in the given ordering) if and only if $j$ is odd.
    
  \end{enumerate}
\end{lemma}
\begin{proof}
  This proof is done in the language of geometric simplicial complexes (rather
  than abstract ones).
  
  The case $\ell=1$ can be done by a stellar subdivision and labelling
  the added apex $x_1$. The case $k=1$ is easy, as illustrated in
  Figure~\ref{f:subdivide} (left). We therefore assume that $k \ge 2$
  and build our subdivision and labelling in four steps:

\begin{itemize}
\item We start with the boundary of our simplex $\sigma$ where each
  vertex $w_i$ is labelled by itself. Let $\vartheta$ be the $(k-1)$-face
  of $\partial \sigma$ opposite vertex $w_2$, \emph{ie} labelled by
  $w_1,w_3,w_4,\cdots w_{k+1}$. We create a vertex in the interior of $\sigma$,   label it $w_2$,
  and construct a new simplex $\sigma'$ as the join of $\vartheta$
  and this new vertex; this is the dark simplex in
  Figure~\ref{f:subdivide} (right).

\item We then subdivide $\sigma'$ by considering $\ell - 1$ distinct
  hyperplanes passing through the vertices of $\sigma'$ labelled $w_3,
  w_4, \ldots, w_{k+1}$ and through an interior points of the edge of
  $\sigma'$ labelled $w_1, w_2$. These hyperplanes subdivide $\sigma'$
  into $\ell$ smaller simplices.
  We label the new interior vertices on the edge of
  $\sigma'$ labelled $w_1, w_2$ by alternatively, $w_1$ and $w_2$;
  since $\ell$ is odd we can do so in a way that every sub-edge is
  bounded by two vertices labelled $w_1, w_2$.

\item We operate a stellar subdivision of each of the $\ell$ smaller
  simplices subdividing $\sigma'$, and label the added
  apices $x_1, x_2, \ldots, x_\ell$. This way we obtain a subdivision $S'$ of
  $\sigma'$.

\item We finally consider each face $\eta$ of $S'$ subdividing $\partial
  \sigma'$ and other than $\vartheta$
  and add the simplex formed by $\eta$ and
  the (original) vertex $w_2$ of $\sigma$. These simplices, together with $S'$,
  form the desired subdivision $S$ of $\sigma$.
\end{itemize}  

\noindent
It follows from the construction that no face of $\partial
\sigma$ was subdivided.

Property~1 is enforced in the first step and preserved
throughout. We can ensure that Property~2 holds in the following way. 
First, we have that any $k$-simplex of $S'$ contains a vertex $x_j$ for some $j
\in [\ell]$. Next, if we consider a $k$-simplex of $S$ which is not in $S'$ it
is a join of a certain $(k-1)$-simplex $\eta$ of $S'$, with $\eta \subset \partial \sigma'$, and the vertex $w_2$ of $\sigma$. 
However, the only such $(k-1)$-simplex labelled by $w_1, w_3, w_4, \dots,
w_{k+1}$ is $\vartheta$, but the join of $\vartheta$ and $w_2$ does not belong
to $S$.

Properties~3 and~4 are enforced by the stellar
subdivisions of the third step and by alternating the labels $w_1$ and $w_2$ in
the second step. No other step creates, destroys or
modifies any simplex involving a vertex labelled $x_j$.
\end{proof}

Let $S$ be the subdivision of a simplex $\sigma$ from Lemma~\ref{l:subdivide}. Similarly as in the case of Lemma~\ref{l:stellar_rho}, we need to describe the
chain map $\rho\colon C_*(\sigma) \to C_*(S)$ defined by formula~\eqref{e:rho}.
Actually, only a partial information will be sufficient for us, focusing on
$k$-simplices of $X_j$.

Since for every $j \in [\ell]$, the apex of $X_j$ is the only vertex labelled by
$x_j$, we can use $x_j$ as the name for the apex. Let $\vartheta_j$ be the 
$k$-simplex on the vertices of $X_j$ except of $x_j$. Note that this simplex
does not belong to $S$. Following the usual pattern, we also denote
$z(\vartheta_j, x_j) := \partial(\vartheta_j \cup \{x_j\})$.

\begin{lemma} 
  \label{l:rho_S}  
  In the setting above, 
  \begin{equation}
    \label{e:rho_S}    
  \rho(\sigma) = \sum\limits_{j=1}^{\ell} (-1)^{j+1}\left(\vartheta_j + (-1)^k z(\vartheta_j, x_j) 
  \right) + \sum\limits_{\eta} \Or(\eta,\sigma) \eta
\end{equation}
where the second sum is over all $k$-simplices of $S$ which do not belong to
any $X_j$.
\end{lemma}

\begin{proof}
  We expand $\rho(\sigma)$ via~\eqref{e:rho}; however, we further shift the
  $k$-simplices in some of the $X_j$ to the first sum in~\eqref{e:rho_S}. This
  is done via Lemma~\ref{l:stellar_rho} on each of the $X_j$; the correction
  term $(-1)^{j+1}$ comes from Property~4 of Lemma~\ref{l:subdivide}.
\end{proof}

The subdivision $D$ of $\skelsim ks$ is now defined as follows. First,
we leave the $(k-1)$-skeleton untouched. Next for each $k$-simplex $\sigma_i$
we consider the multipoint $\mu = \mu_{c(i)} = \sum_{u \in U} \lambda_u
u$ (leaving 
the dependence on the index $i$ implicit in the affine combination). We
recall that $\lambda_u$ are elements of $\Z_p$; however, we temporarily
consider them as elements of $\Z$, in the interval $\{0, 1, \dots, p-1\}$. We
consider some $u' \in U$, which belongs to the support of $\mu$, and we set $\kappa_u := \lambda_u$ for any $u \in U
\setminus \{u'\}$ (as elements of $\Z$) whereas we set $\kappa_{u'} := 1 -
\sum_{u \in U \setminus \{u'\}} \lambda_u$. It follows that $\kappa_u \equiv
\lambda_u \pmod p$ for any $u \in U$ as $\sum_{u \in U} \lambda_u \equiv 1
\pmod p$ (they sum to $1$ as elements of $\Z_p$). Next, we set $\ell_i :=
\sum_{u \in U} |\kappa_u|$. It follows that $\ell_i$ is odd, and we set $S(i)$
to be the subdivision of $\sigma_i$ obtained from Lemma~\ref{l:subdivide} with
$\ell := \ell_i$. The final subdivision $D$ is obtained by subdividing each
$\sigma_i$ this way. For working with the chains, we need to specify a global
linear order on the vertices $D$. We pick an arbitrary such order that respects
the prescribed order on each $S(i)$.

According to this subdivision, we have a chain map $\rho  \colon C_*(\skelsim
ks) \to C_*(D)$ defined in Subsection~\ref{ss:subdivisions}. On faces of
dimension at most $(k-1)$ it is an identity; on $k$-faces, it is determined by
the formula from Lemma~\ref{l:rho_S}.

\paragraph{The simplicial map $\gsimp$.}

We now define a simplicial map $\gsimp\colon D \to \skelsim kn$. We
first set $\gsimp(v) = v$ for every vertex $v$ of $\simplex s$.
Next, we consider some $k$-face $\sigma_i = \{w_1, \dots, w_{k+1}\}$.  We
denote by $v_1, v_2 \ldots, v_{k+1}$ the
vertices on the boundary of $S(i)$, being understood that each
$v_j$ is labelled by $w_j$. We map each interior vertex of $S(i)$ labelled
with $w_j$ to $w_j$. It remains to map interior vertices of $S(i)$ labelled
$x_j$ for $j \in [\ell]$. Using the notation from the definition of $D$, we
consider the integers $\kappa_u$ for $u \in U$ (with respect to our
$\sigma_i$). If $\kappa_u > 0$, then we pick $\kappa_u$ vertices $x_j$ with $j$
odd and we map them to $u$. If $\kappa_u < 0$, which may happen only for $u =
u'$ (coming again from the definition of $D$), then we pick $-\kappa_u$
vertices $x_j$ with $j$ even and we map them to $u$. Of course, for two
distinct elements $u_1$ and $u_2$ from $U$ we pick distinct points $x_j$. The
parameter $\ell = \ell_i$ is set up exactly in such a way that we cover all
$x_j$. Now we need to know that $\rho$ and $\gsimp$ compose to $\varphi$ on the
level of chains.

\begin{lemma}
$(\gsimp)_{\sharp} \circ \rho = \varphi$.
\end{lemma}

\begin{proof}
All three maps are the identity on $\skelsim{k-1}s$ so let us focus
on the $k$-faces. Consider a $k$-face $\sigma_i$, the value $\rho(\sigma_i)$ is
given by the formula in Lemma~\ref{l:rho_S} with $S = S(i)$. However, for expressing
$(\gsimp)_{\sharp} \circ \rho(\sigma_i)$ we may ignore the second sum in
formula~\eqref{e:rho_S} since a $k$-simplex $\eta$ of $S$ that does not belong to
any $X_j$ contains two vertices with a same label by Lemma~\ref{l:subdivide},
which implies that $(\gsimp)_{\sharp}(\eta) = 0$.

Therefore
\begin{equation}
\label{e:g_rho}
(\gsimp)_{\sharp} \circ \rho(\sigma_i) = 
(\gsimp)_{\sharp} \Bigl(
\sum\limits_{j=1}^{\ell} (-1)^{j+1}\left(\vartheta_j + (-1)^k z(\vartheta_j, x_j) 
\right) \Bigr) = \sum\limits_{u \in U} \kappa_u(\sigma_i + (-1)^k z (\sigma_i,
u)).
\end{equation}

The last equality follows from the definition of $\gsimp$ considering that
$\gsimp$ preserves the prescribed linear orders on $D$ and $\skelsim kn$. The
sign $(-1)^{j+1}$ disappears as the vertices $x_j$ with $j$ even contribute to
$\kappa_u$ with the opposite sign. We know that $\kappa_u \mod p = \lambda_u$
and that $\sum_{u \in U} \kappa_u = 1$. Therefore the expression on the
right-hand side of~\eqref{e:g_rho} equals $\sigma_i + (-1)^k z(\sigma_i, \mu)$,
that is, $\varphi(\sigma_i)$ as required.
\end{proof}

\paragraph{The continuous map $g$.}

Since $D$ is a subdivision of $\skelsim ks$, we have $|\skelsim ks| =
|D|$ and the simplicial map $\gsimp\colon D \to \skelsim kn$ induces a
continuous map $g \colon |\skelsim ks| \to |\skelsim kn|$. All that
remains to do is check that $g$ satisfies the two conditions of
Lemma~\ref{l:chain_p}. Condition~1 follows from a direct translation of
Lemma~\ref{l:chain_ae}; note that in the definition of $\gsimp$ we map $x_j$ to
$u \in U$ only if $\kappa_u \neq 0$. Condition~2 can be verified by a computation
in the same way as in Section~\ref{s:weak}. Specifically, in homology we have
$$ f_* \circ \varphi_* = f_* \circ (\gsimp)_* \circ \rho_*$$
and we know that $f_* \circ \varphi_*$ is trivial on $\skelsim ks$ by
Lemma~\ref{l:zero_homology_multipoints}. As $\rho_*$ is an
isomorphism, this implies that $f_* \circ (\gsimp)_*$ is trivial.
Lemma~\ref{l:commutative_diagram} then implies that $(f \circ g)_*$ is
trivial. This concludes the proof of Lemma~\ref{l:chain_p}.

\bigskip

\begin{acknowledgement} U.W.\ learned about Conjecture~\ref{c:kuhnel}
  from Wolfgang K\"uhnel when attending the \emph{Mini Symposia on
    Discrete Geometry and Discrete Topology} at the \emph{Jahrestagung
    der Deutschen Mathematiker-Vereinigung} in M\"unchen in 2010. He
  would like to thank the organizers Frank Lutz and Achill Sch\"urmann
  for the invitation, and Prof.~K\"uhnel for stimulating discussions.
\end{acknowledgement}

\bibliographystyle{alpha}
\bibliography{kuhnel}

\clearpage
\appendix

\end{document}